\documentclass[3p,times,reqno]{elsarticle}

\usepackage{ecrc-1}

\usepackage{enumitem}

\firstpage{1}

\runauth{H. Chen and U.N. Katugampola}

\usepackage{amssymb}
 \usepackage{amsthm}
\usepackage{hyperref} 
\usepackage{color,soul} 
\usepackage[normalem]{ulem}  
\usepackage{verbatim} 
\usepackage{graphicx,subfig,epstopdf}  
\usepackage{epsfig,epstopdf,verbatim,subfig,caption}
\usepackage[english]{babel}
\usepackage[export]{adjustbox}%
\usepackage[utf8]{inputenc}
\usepackage{amsmath}
\usepackage{graphicx}
\usepackage[colorinlistoftodos]{todonotes}
\usepackage{lmodern}
\usepackage{amssymb,amsthm}
\usepackage{thmtools}
\usepackage{verbatim}
\usepackage{chemarr}
\usepackage{epsfig}
\usepackage{psfrag}
\usepackage{float}
\usepackage{caption}

\newcommand{\R}{\mathbb{R}}

\newcommand{\K}{{\mathcal K}}
\newcommand{\G}{{\mathcal G}}

 \biboptions{sort&compress}
\newtheorem{theorem}{Theorem}[section]
\newtheorem{lemma}[theorem]{Lemma}
\newtheorem{corollary}[theorem]{Corollary}
\theoremstyle{definition}
\newtheorem{definition}[theorem]{Definition}

\theoremstyle{remark}
\newtheorem{remark}[theorem]{Remark}

\begin{document}
\begin{frontmatter}
 \title{\normalsize{\textit{Dedicated to Prof. Richard M. Aron for his contributions to Mathematics.}} \\ \vspace{.4cm} \Large{Hermite-Hadamard and Hermite-Hadamard-Fej\'{e}r type Inequalities for Generalized Fractional Integrals}\tnoteref{fn1}} \normalsize


\author{Hua Chen}
\ead{chenhua@udel.edu}
\author{Udita N. Katugampola\corref{corresp}}
\ead{uditanalin@yahoo.com}
%

\cortext[corresp]{Corresponding author.  Tel.: +13028312694.}

%
\address{Department of Mathematical Sciences, University of Delaware, Newark, DE 19716, U.S.A.}
\begin{abstract}
In this paper we obtain the Hermite-Hadamard and Hermite-Hadamard-Fej\'{e}r type inequalities for fractional integrals which generalize the two familiar fractional integrals namely, the Riemann-Liouville and the Hadamard fractional integrals into a single form. We prove that, in most cases, we obtain the Riemann--Liouville and the Hadamard equivalence just by taking limits when a parameter $\rho \rightarrow 1$ and $\rho \rightarrow 0^+$, respectively. 
\end{abstract}
\begin{keyword}
Hermite-Hadamard Inequalities \sep Hermite-Hadamard-Fej{\'e}r inequalities \sep Riemann-Liouville fractional integral \sep Hadamard fractional integral \sep Katugampola fractional integral  \sep convexity
%
%
\MSC[2010] 26A33 \sep 26D10 \sep 26D15
\end{keyword}
\end{frontmatter}
%

%
\section{Introduction}
%
%

%
%

The classical Hermite–Hadamard inequality provides estimates of the mean value of a continuous convex function $f:[a, b]\rightarrow \mathbb{R}$. The function $f:[a, b] \subset\mathbb{R} \rightarrow \mathbb{R}$, is said to be convex if the following inequality holds
\[
    f(\lambda x + (1 - \lambda)y) \leq \lambda f(x) + (1 - \lambda)f(y),
\] 
for all $x, y \in [a, b]$ and $\lambda \in [0, 1]$. We say that $f$ is concave if $(-f)$ is convex.

Let $f: I \to \R $ be a convex function defined on the interval $I$ of real numbers and $a, b \in I$ with $a < b$, then 
\begin{eqnarray}
\label{HHineq}
f\left(\frac{a+b}{2}\right) \leq \frac{1}{b-a} \int_a^b f(x)dx \leq \frac{f(a)+f(b)}{2},
\end{eqnarray}
which is known as the Hermite--Hadamard inequality \cite{hadamard1}. In \cite{fejer1},  Fej{\'e}r developed the weighted generalization of the Hermite--Hadamard inequality given below.
\begin{theorem}
\label{thmfejer}
Let $f:[a, b] \to \R$ be a convex function. Then the inequality 
\begin{equation}
\label{HHFineq}
f\left(\frac{a+b}{2}\right)\int_a^b g(x)dx \leq \frac{1}{b-a} \int_a^b f(x)g(x)dx \leq \frac{f(a)+f(b)}{2} \int_a^b g(x)dx
\end{equation}
holds, where $g:[a, b] \to \R$ is non-negative, integrable and symmetric to $(a+b)/2$.
\end{theorem}

Since then, many researches generalized and extended the two inequalities \eqref{HHineq} and \eqref{HHFineq}. For related results, for example, see \cite{HH1,HH2,HH3,HH4,HH5,HH6,HH7,HH8,HH9} and the references therein. In \cite{sari1}, Sarikaya et al. generalized the Hermite--Hadamard type inequalities via Riemann--Liouville fractional integrals. Then in \cite{Iscan1}, {\.I}{\c{s}}can extended Sarikaya's results to Hermite--Hadamard--Fej{\'e}r type inequalities for fractional integrals. Further results involving the two inequalities in question with applications to fractional integrals can be found, for example, in \cite{sari1,HH11,HH12,HH13} and the references therein. 

In \cite{udita1}, the second author introduces an Erd\'{e}lyi-Kober type fractional integral operator and uses that integral to define a new fractional derivative in \cite{udita2}, which generalizes the Riemann-Liouville and the Hadamard fractional derivatives to a single form and argued that it is not possible to derive the Hadamard equivalence operators from the corresponding Erd\'{e}lyi-Kober type operators, thus making the new derivative more appropriate for modeling certain phenomena which undergo bifurcation-like behaviors. For further properties of the Erd\'{e}lyi-Kober operators, the interested reader is refereed to, for example, \cite{Kiry,KilSri,Samko}. According to the literature, the newly defined fractional operators are know known as the \textit{Katugampola fractional integral} and \textit{derivatives}, respectively. For consistency, we use the same name for those operators in question. It can be shown that the derivatives in question satisfy the fractional derivative criteria (test) given in \cite{what2,corr1}. These operators have applications in fields such as in probability theory \cite{prob1}, theory of inequalities \cite{ineq1,HH1,HH2}, variational principle \cite{vara1}, numerical analysis \cite{numr1}, and Langevin equations \cite{lang1}. A Caputo-type modification of the operator in question can be found in \cite{caka}. The interested reader is referred, for example, to \cite{udita4,u-1,u-3,u-4,u-5,u-6,u-7,u-8,u-9,u-10} for further results on these and similar operators. The Mellin transforms of the generalized fractional integrals and derivatives defined in \cite{udita1} and \cite{udita2}, respectively, are given in \cite{udita3}. The same reference also studies a class of sequences that are closely related to the Stirling numbers of the $2^{nd}$ kind. The $\rho-$Laplace and $\rho-$Fourier transforms of the Katugampola fractional operators are given in \cite{rLF}.

In the following, we will give some necessary definitions and preliminary results which are used and referred to throughout this paper.

\begin{definition}[\cite{pod}] 
Let $\alpha >0$ with $n-1 < \alpha \leq n, \; n \in \mathbb{N}$, and $a < x<b$. The left- and right-side Riemann--Liouville fractional integrals of order $\alpha$ of a function $f$ are given by 
\[
     J^\alpha_{a+}f(x) = \frac{1}{\Gamma(\alpha)}\int_{\rm a}^x (x-t)^{\alpha -1} f(t) \, dt  \quad \mbox{and} \quad J^\alpha_{b-}f(x) = \frac{1}{\Gamma(\alpha)}\int_x^b (t-x)^{\alpha -1} f(t) \, dt 
\]
\noindent respectively, where $\Gamma(\cdot)$ is the Euler's gamma function defined by
$$
\Gamma(x) = \int_0^\infty t^{x-1}e^{-t} \, dt.
$$
\end{definition}

\begin{definition}[\cite{Samko}] 
Let $\alpha >0$ with $n-1 < \alpha \leq n, \; n \in \mathbb{N}$, and $a < x<b$. The left- and right-side Hadamard fractional integrals of order $\alpha$ of a function $f$ are given by 
\begin{equation}\label{Had1}
    H^\alpha_{a+}f(x) = \frac{1}{\Gamma(\alpha)}\int_{\rm a}^x \left(\ln \frac{x}{t} \right)^{\alpha -1} \frac{f(t)}{t} \, dt \quad \mbox{and} \quad H^\alpha_{b-}f(x) = \frac{1}{\Gamma(\alpha)}\int^{\rm b}_x \left(\ln \frac{t}{x} \right)^{\alpha -1} \frac{f(t)}{t} \, dt .
\end{equation}
\end{definition}

In \cite{sari1}, Sarikaya et al. established the Hermite--Hadamard inequalities via the Riemann--Liouville fractional integrals as follows.
\begin{theorem} \label{sarikaya_thm1}
Let $f:[a, b]\to \R$ be a positive function with $0\leq a <b$ and $f\in L[a, b]$. If $f$ is a convex function on $[a, b]$, then the following inequalities hold 
\begin{equation} \label{Sari_HH1}
f\left(\frac{a+b}{2}\right) \leq \frac{\Gamma(\alpha+1)}{2(b-a)^\alpha}[J^\alpha_{a+}f(b) +J^\alpha_{b-}f(a) ] \leq \frac{f(a)+f(b)}{2}
\end{equation}
with $\alpha>0$.
\end{theorem}

\begin{theorem} \label{sarikaya_thm2}
Let $f:[a, b]\to \R$ be a differentiable mapping on $(a, b)$ with $a<b$. If $|f'|$ is a convex function on $[a, b]$, then the following inequality holds for $\alpha>0$,
\begin{equation} \label{Sari_HH2}
\left| \frac{f(a)+f(b)}{2}-\frac{\Gamma(\alpha+1)}{2(b-a)^\alpha}[J^\alpha_{a+}f(b) +J^\alpha_{b-}f(a) ] \right| \leq \frac{b-a}{2(\alpha+1)}\left( 1-\frac{1}{2^\alpha}\right)[|f'(a)|+|f'(b)|].
\end{equation}
\end{theorem}

Further, in \cite{Iscan1}, {\.I}{\c{s}}can extended these results to Hermite--Hadamard--Fej{\'e}r type inequalities as follows.

\begin{theorem} \label{Iscan_thm1}
Let $f:[a, b]\to \R$ be a convex function with $a<b$ and $f\in L[a,b]$. If $g:[a, b]\to \R$ is non-negative, integrable and symmetric to $(a+b)/2$, then the following inequalities for fractional integrals hold
\begin{equation} \label{Iscan_HH1}
f\left(\frac{a+b}{2}\right)[J^\alpha_{a+}g(b)+J^\alpha_{b-}g(a)] \leq [J^\alpha_{a+}(gf)(b) +J^\alpha_{b-}(gf)(a) ] \leq \frac{f(a)+f(b)}{2} [J^\alpha_{a+}g(b)+J^\alpha_{b-}g(a)]
\end{equation}
with $\alpha>0$.
\end{theorem}

\begin{theorem} \label{Iscan_thm2}
Let $f:[a, b]\to \R$ be a differentiable mapping on $(a, b)$ and $f'\in L[a,b]$ with $a<b$. If $|f'|$ is convex on $[a, b]$ and $g:[a, b]\to \R$ is continuous and symmetric to $(a+b)/2$, then the following inequality holds
\begin{equation} \label{Iscan_HH2}
\left| \frac{f(a)+f(b)}{2} [J^\alpha_{a+}g(b)+J^\alpha_{b-}g(a)]- [J^\alpha_{a+}(gf)(b) +J^\alpha_{b-}(gf)(a) ] \right| \leq \frac{(b-a)^{\alpha+1}||g||_{\infty}}{(\alpha+1)\Gamma(\alpha+1)}\left( 1-\frac{1}{2^\alpha}\right)[|f'(a)|+|f'(b)|]
\end{equation}
with $\alpha>0$, where $||g||_\infty = \sup_{t\in[a, b]}|g(x)|$.
\end{theorem}

Recently, Katugampola introduced a new fractional integral that generalizes the Riemann--Liouville and the Hadamard fractional integrals into a single form (see  \cite{udita1, udita2, udita3}). The purpose of this paper is to derive Hermite--Hadamard type and Hermite--Hadamard--Fej{\'e}r type inequalities using the Katugampola fractional integrals. Since it is a generalization of Hadamard fractional integral, we can also get the inequalities for Hadamard fractional integral in some cases by just taking limits, while we obtain Riemann-Liouville equivalence by taking limits in all the cases. 

\begin{definition}[\cite{udita2}] 
Let $[a, b] \subset \mathbb{R}$ be a finite interval. Then, the  left- and right-side Katugampola fractional integrals of order $\alpha \,  (>0)$ of $f \in X^p_c(a, b)$ are defined by \cite{udita2},
\[
     ^\rho I^\alpha_{a+}f(x) = \frac{\rho^{1-\alpha}}{\Gamma(\alpha)}\int_{\rm a}^x \frac{t^{\rho-1}}{(x^\rho-t^\rho)^{1-\alpha}} f(t) \, dt  \quad \mbox{and} \quad ^\rho I^\alpha_{b-}f(x) = \frac{\rho^{1-\alpha}}{\Gamma(\alpha)}\int^{\rm b}_x \frac{t^{\rho-1}}{(t^\rho-x^\rho)^{1-\alpha}} f(t) \, dt 
\]
with $a<x<b$ and $\rho>0$, if the integrals exist.
\end{definition}

\begin{theorem}[\cite{udita2}]
Let $\alpha >0$ and $\rho >0$. Then for $x>a$,
\begin{enumerate}
\item $\displaystyle \lim_{\rho \to 1} \,^\rho I^\alpha_{a+}f(x) = J^\alpha_{a+}f(x)$,
\item $\displaystyle \lim_{\rho \to 0^+} \,^\rho I^\alpha_{a+}f(x) = H^\alpha_{a+}f(x)$.
\end{enumerate}
Similar results also hold for right-sided operators.
\end{theorem}

\section{Main Results}

First we generalize Sarikaya's results \cite{sari1} of the Hermite-Hadamard's inequalities for the Katugampola fractional integrals. 

\begin{theorem}\label{th:hh1}
Let $\alpha > 0$ and $\rho>0$. Let $f:[a^\rho, b^\rho] \rightarrow \mathbb{R}$ be a positive function with $0 \leq a < b$ and $f \in X^p_c(a^\rho, b^\rho)$. If $f$ is also a convex function on $[a, b]$, then the following inequalities hold:
\begin{equation}
\label{HH_eq1}
f\left(\frac{a^{\rho}+b^{\rho}}{2}\right)\leq \frac{\rho^\alpha \Gamma(\alpha+1)}{2(b^\rho-a^\rho)^\alpha}\left[ {}^\rho I^\alpha_{a+}f(b^{\rho})+ {}^\rho I^\alpha_{b-}f(a^{\rho})\right] \leq \frac{f(a^{\rho})+f(b^{\rho})}{2}
\end{equation}
where the fractional integrals are considered for the function $f(x^{\rho})$ and evaluated at $a$ and $b$, respectively. 
\end{theorem}

\begin{proof} Let $t\in[0, 1]$. Consider $x, y \in [a, b]$, $a \geq 0$, defined by $x^{\rho} = t^{\rho}a^{\rho} +(1-t^{\rho})b^{\rho}$, $y^{\rho} = (1-t^{\rho})a^{\rho} +t^{\rho}b^{\rho}$. Since $f$ is a convex function on $[a, b]$, we have
\[
   f\left(\frac{x^{\rho}+y^{\rho}}{2}\right)\leq \frac{f(x^{\rho})+f(y^{\rho})}{2}
\]
Then we have
\begin{equation}\label{eq1}
  2f\left(\frac{a^{\rho}+b^{\rho}}{2}\right) \leq f(t^{\rho}a^{\rho} +(1-t^{\rho})b^{\rho}) + f((1-t^{\rho})a^{\rho} +t^{\rho}b^{\rho})
\end{equation}
Multiplying both sides of Eq.~(\ref{eq1}) by $t^{\alpha\rho-1}, \, \alpha > 0$ and then integrating the resulting inequality with respect to $t$ over $[a, b]$, we obtain
\begin{align}
  \frac{2}{\alpha\rho}f\left(\frac{a^{\rho}+b^{\rho}}{2}\right) &\leq \int_0^1 t^{\alpha\rho-1}f(t^{\rho}a^{\rho} +(1-t^{\rho})b^{\rho})\,dt + \int_0^1 t^{\alpha\rho-1}f((1-t^{\rho})a^{\rho} +t^{\rho}b^{\rho})\,dt \label{eq10}\\
	&=\int_b^a\left(\frac{b^{\rho}-x^{\rho}}{b^{\rho}-a^{\rho}}\right)^{\alpha-1}f(x^{\rho})\frac{x^{\rho-1}}{a^{\rho}-b^{\rho}}\,dx 
	   + \int_a^b\left(\frac{y^{\rho}-a^{\rho}}{b^{\rho}-a^{\rho}}\right)^{\alpha-1}f(y^{\rho})\frac{y^{\rho-1}}{b^{\rho}-a^{\rho}}\,dy \nonumber\\		
	&= 	\frac{\rho^{\alpha-1} \Gamma(\alpha+1)}{(b^\rho-a^\rho)^\alpha}\left[ {}^\rho I^\alpha_{a+}f(b^{\rho})+ {}^\rho I^\alpha_{b-}f(a^{\rho})\right]. \label{eq11}
\end{align}
This establishes the first inequality. For the proof of the second inequality in Eq.~(\ref{HH_eq1}), we first note that for a convex function $f$, we have
\[
f(t^{\rho}a^{\rho} +(1-t^{\rho})b^{\rho}) \leq t^{\rho}f(a^{\rho})+(1-t^{\rho})f(b^{\rho}),
\]
and
\[
f((1-t^{\rho})a^{\rho} +t^{\rho}b^{\rho}) \leq (1-t^{\rho})f(a^{\rho})+t^{\rho}f(b^{\rho}).
\]
By adding these inequalities, we then have
\begin{equation}\label{eq2}
  f(t^{\rho}a^{\rho} +(1-t^{\rho})b^{\rho}) + f((1-t^{\rho})a^{\rho} +t^{\rho}b^{\rho}) \leq f(a^{\rho})+f(b^{\rho}).
\end{equation}
Multiplying both sides of Eq.~(\ref{eq2}) by $t^{\alpha\rho-1}, \, \alpha > 0$ and then integrating the resulting inequality with respect to $t$ over $[a, b]$, we similarly obtain
\[
   \frac{\rho^{\alpha-1} \Gamma(\alpha)}{(b^\rho-a^\rho)^\alpha}\left[ {}^\rho I^\alpha_{a+}f(b^{\rho})+ {}^\rho I^\alpha_{b-}f(a^{\rho})\right] \leq \frac{f(a^{\rho})+f(b^{\rho})}{\alpha\rho}.
\]
This completes the proof of the Theorem~\ref{th:hh1}. \end{proof}

If the function $f^\prime$ is differentiable, we have the following result.
\begin{theorem}\label{th5} Let $f:[a^\rho, b^\rho] \rightarrow \mathbb{R}$ be a differentiable mapping with $0 \leq a < b$. If $f^\prime$ is differentiable on $(a^\rho, b^\rho)$, then the following inequality holds:
\begin{equation}
   \left|\frac{f(a^\rho)+f(b^\rho)}{2} - \frac{\alpha\rho^\alpha \Gamma(\alpha+1)}{2(b^\rho-a^\rho)^\alpha}\left[ {}^\rho I^\alpha_{a+}f(b^{\rho})+ {}^\rho I^\alpha_{b-}f(a^{\rho})\right] \right| \leq \frac{(b^\rho - a^\rho)^2}{2(\alpha+1)(\alpha+2)}\Big(\alpha+\frac{1}{2^\alpha}\Big)\sup_{\xi\in[a^\rho, b^\rho]}\big|f^{\prime\prime}(\xi)\big|.
\end{equation}
\end{theorem}
\begin{proof}Using right side of inequality (\ref{eq10}) and Eq.~(\ref{eq11}), we have 
\[
  \frac{\rho^{\alpha-1} \Gamma(\alpha+1)}{(b^\rho-a^\rho)^\alpha}\left[ {}^\rho I^\alpha_{a+}f(b^{\rho})+ {}^\rho I^\alpha_{b-}f(a^{\rho})\right] = \int_0^1 t^{\alpha\rho-1}f(t^{\rho}a^{\rho} +(1-t^{\rho})b^{\rho})\,dt + \int_0^1 t^{\alpha\rho-1}f((1-t^{\rho})a^{\rho} +t^{\rho}b^{\rho})\,dt.       
\]
By using integration by parts, we then have
\begin{align}\label{eqx}
   \frac{f(a^\rho)+f(b^\rho)}{\alpha\rho} &- \frac{\rho^{\alpha-1}\Gamma(\alpha+1)}{(b^\rho-a^\rho)^\alpha}\left[ {}^\rho I^\alpha_{a+}f(b^{\rho})+ {}^\rho I^\alpha_{b-}f(a^{\rho})\right] \nonumber \\
	&\hspace{1.9cm}= \frac{b^\rho - a^\rho}{\alpha}\int_0^1 t^{\rho(\alpha+1)-1}\left[f^\prime\left((1-t^\rho)a^\rho+t^\rho b^\rho\right)-f^\prime\left(t^\rho a^\rho+(1-t^\rho)b^\rho\right)\right]\,dt.
\end{align}
Using Eq.~(\ref{eqx}) and applying the mean value theorem for the function $f^\prime$, we have
\[
\frac{f(a^\rho)+f(b^\rho)}{\alpha\rho} - \frac{\rho^{\alpha-1}\Gamma(\alpha+1)}{(b^\rho-a^\rho)^\alpha}\left[ {}^\rho I^\alpha_{a+}f(b^{\rho})+ {}^\rho I^\alpha_{b-}f(a^{\rho})\right] = \frac{(b^\rho - a^\rho)^2}{\alpha}\int_0^1 t^{\rho(\alpha+1)-1}\big(2t^\rho-1\big)f^{\prime\prime}(\xi(t))\,dt, 
\]
where $\xi(t)\in(a^\rho, b^\rho)$. This leads us to
\begin{align*}
  \bigg|\frac{f(a^\rho)+f(b^\rho)}{\alpha\rho} &- \frac{\rho^{\alpha-1}\Gamma(\alpha+1)}{(b^\rho-a^\rho)^\alpha}\left[ {}^\rho I^\alpha_{a+}f(b^{\rho})+ {}^\rho I^\alpha_{b-}f(a^{\rho})\right]\bigg| \\
	  &\leq \frac{(b^\rho - a^\rho)^2}{\alpha}\int_0^1 t^{\rho(\alpha+1)-1}\big|2t^\rho-1\big|\big|f^{\prime\prime}(\xi(t))\big|\,dt \\
		&\leq \frac{(b^\rho - a^\rho)^2}{\alpha}\sup_{\xi\in[a^\rho, b^\rho]}\big|f^{\prime\prime}(\xi)\big|\Big[\int_0^{\frac{1}{\sqrt[\rho]{2}}}\big(1-2t^\rho\big) t^{\rho(\alpha+1)-1}\,dt+ \int_{\frac{1}{\sqrt[\rho]{2}}}^1\big(2t^\rho-1\big) t^{\rho(\alpha+1)-1}\,dt\bigg]\\
		&=\frac{(b^\rho - a^\rho)^2}{\alpha\rho(\alpha+1)(\alpha+2)}\Big(\alpha+\frac{1}{2^\alpha}\Big)\sup_{\xi\in[a^\rho, b^\rho]}\big|f^{\prime\prime}(\xi)\big|.
\end{align*}
This gives the desired result.
\end{proof}

If $|f^\prime|$ is also convex on $[a^\rho, b^\rho]$, we then have the following result.

\begin{theorem}\label{th4} Let $f:[a^\rho, b^\rho] \rightarrow \mathbb{R}$ be a differentiable mapping on $(a^\rho, b^\rho)$ with $0 \leq a < b$. If $|f^\prime|$ is convex on $[a^\rho, b^\rho]$, then the following inequality holds:
\begin{equation}
   \left|\frac{f(a^\rho)+f(b^\rho)}{2} - \frac{\alpha\rho^\alpha \Gamma(\alpha+1)}{2(b^\rho-a^\rho)^\alpha}\left[ {}^\rho I^\alpha_{a+}f(b^{\rho})+ {}^\rho I^\alpha_{b-}f(a^{\rho})\right] \right| \leq \frac{b^\rho - a^\rho}{2(\alpha+1)}\left[|f^\prime(a^\rho)|+|f^\prime(b^\rho)|\right].
\end{equation}
\end{theorem}
\begin{proof}By using Eq.~(\ref{eqx}), the triangle inequality and the convexity of $|f^\prime|$, we get
\begin{align*}
  \Bigg|\frac{f(a^\rho)+f(b^\rho)}{\alpha\rho} &- \frac{\rho^{\alpha-1}\Gamma(\alpha+1)}{(b^\rho-a^\rho)^\alpha}\left[ {}^\rho I^\alpha_{a+}f(b^{\rho})+ {}^\rho I^\alpha_{b-}f(a^{\rho})\right]\Bigg|  \\
	      &\leq \frac{b^\rho - a^\rho}{\alpha}\int_0^1 t^{\rho(\alpha+1)-1}\Big|f^\prime\left((1-t^\rho)a^\rho+t^\rho b^\rho\right)-f^\prime\left(t^\rho a^\rho+(1-t^\rho)b^\rho\right)\Big|\,dt\\
				&\leq \frac{b^\rho - a^\rho}{\alpha}\int_0^1 t^{\rho(\alpha+1)-1}\Big[\left|f^\prime\left((1-t^\rho)a^\rho+t^\rho b^\rho\right)\right|+\left|f^\prime\left(t^\rho a^\rho+(1-t^\rho)b^\rho\right)\right|\Big]\,dt\\
				&\leq \frac{b^\rho - a^\rho}{\alpha}\int_0^1 t^{\rho(\alpha+1)-1}\Big[(1-t^\rho)|f^\prime(a^\rho)|+t^\rho|f^\prime(b^\rho)|+t^\rho|f^\prime( a^\rho)|+(1-t^\rho)|f^\prime(b^\rho)|\Big]\,dt \\
				&= \frac{b^\rho - a^\rho}{\alpha}\Big[|f^\prime(a^\rho)|+f^\prime(b^\rho)|\Big]\int_0^1 t^{\rho(\alpha+1)-1}\,dt\\
				&= \frac{b^\rho - a^\rho}{\alpha\rho(\alpha+1)}\Big[|f^\prime(a^\rho)|+f^\prime(b^\rho)|\Big].
\end{align*}
This establishes the result. 
\end{proof} 
Another more strict inequality can be obtain by using the following lemma.
\begin{lemma}\label{lem1} Let $f:[a^\rho, b^\rho] \rightarrow \mathbb{R}$ be a differentiable mapping on $(a^\rho, b^\rho)$ with $0 \leq a < b$. Then the following equality holds if the fractional integrals exist:
\begin{equation}
   \frac{f(a^\rho)+f(b^\rho)}{2} - \frac{\alpha\rho^\alpha \Gamma(\alpha+1)}{2(b^\rho-a^\rho)^\alpha}\left[ {}^\rho I^\alpha_{a+}f(b^{\rho})+ {}^\rho I^\alpha_{b-}f(a^{\rho})\right] = \frac{b^\rho - a^\rho}{2}\int_0^1\left[(1-t^\rho)^\alpha - t^{\rho\alpha}\right]t^{\rho-1}f^\prime\big(t^\rho a^\rho+(1-t^\rho)b^\rho\big)\,dt.
\end{equation}
\end{lemma}
\begin{proof}This can be proved using a similar line of argument as in the proof of Lemma 2 in \cite{sari1}. To that end, by integration by parts, first note that
\begin{align*}
   \int_0^1 \big(1-t^\rho)^\alpha t^{\rho-1}&f^\prime\big(t^\rho a^\rho+(1-t^\rho)b^\rho\big)\,dt \\
	   & = \frac{(1-t^\rho)^\alpha f\big(t^\rho a^\rho+(1-t^\rho)b^\rho\big)}{\rho(a^\rho-b^\rho)}\bigg|_0^1 + \frac{\alpha}{a^\rho-b^\rho}\int_0^1(1-t^\rho)^{\alpha-1}t^{\rho-1}f\big(t^\rho a^\rho+(1-t^\rho)b^\rho\big)\,dt\\
		 & = \frac{f(b^\rho)}{\rho(b^\rho-a^\rho)}-\frac{\alpha}{b^\rho-a^\rho}\int_b^a\Big(\frac{x^\rho-a^\rho}{b^\rho-a^\rho}\Big)^{\alpha-1}\cdot\frac{x^{\rho-1}}{a^\rho-b^\rho}\,dx \\
		& = \frac{f(b^\rho)}{\rho(b^\rho-a^\rho)} - \frac{\alpha\rho^{\alpha-1} \Gamma(\alpha+1)}{(b^\rho-a^\rho)^{\alpha+1}}{}^\rho I^\alpha_{b-}f(x^{\rho})\bigg|_{x=a}.
\end{align*}
Similarly, we can also prove that
\[
   -\int_0^1 t^{\rho\alpha}\cdot t^{\rho-1}f^\prime\big(t^\rho a^\rho+(1-t^\rho)b^\rho\big)\,dt = \frac{f(a^\rho)}{\rho(b^\rho-a^\rho)} - \frac{\alpha\rho^{\alpha-1} \Gamma(\alpha+1)}{(b^\rho-a^\rho)^{\alpha+1}}{}^\rho I^\alpha_{a+}f(x^{\rho})\bigg|_{x=b}.
\]
These two results lead to the proof of Lemma~\ref{lem1}.
\end{proof}
With the help of this lemma, we have the following result. 

\begin{theorem}\label{th4}Let $f:[a^\rho, b^\rho] \rightarrow \mathbb{R}$ be a differentiable mapping on $(a^\rho, b^\rho)$ with $0 \leq a < b$. If $|f^\prime|$ is convex on $[a^\rho, b^\rho]$, then the following inequality holds:
\begin{equation}
   \bigg|\frac{f(a^\rho)+f(b^\rho)}{2} - \frac{\alpha\rho^\alpha \Gamma(\alpha+1)}{2(b^\rho-a^\rho)^\alpha}\left[ {}^\rho I^\alpha_{a+}f(b^{\rho})+ {}^\rho I^\alpha_{b-}f(a^{\rho})\right]\bigg| \leq \frac{b^\rho - a^\rho}{2\rho(\alpha+1)}\Big(1-\frac{1}{2^\alpha}\Big)\big[|f^\prime(a^\rho)|+|f^\prime(b^\rho)|\big].
\end{equation}
\end{theorem}
\begin{proof}Using Lemma~\ref{lem1} and the convexity of $|f^\prime|$, we have
\begin{align*}
   \bigg|\frac{f(a^\rho)+f(b^\rho)}{2} &- \frac{\alpha\rho^\alpha \Gamma(\alpha+1)}{2(b^\rho-a^\rho)^\alpha}\left[ {}^\rho I^\alpha_{a+}f(b^{\rho})+ {}^\rho I^\alpha_{b-}f(a^{\rho})\right]\bigg| \\
	      &\hspace{-1cm}\leq \frac{b^\rho - a^\rho}{2}\int_0^1 t^{\rho-1}\left|(1-t^\rho)^\alpha - t^{\rho\alpha}\right|\big|f^\prime\big(t^\rho a^\rho+(1-t^\rho)b^\rho\big)\big|\,dt\\
	      &\hspace{-1cm}\leq \frac{b^\rho - a^\rho}{2}\int_0^1 t^{\rho-1}\left|(1-t^\rho)^\alpha - t^{\rho\alpha}\right|\big[t^\rho\big|f^\prime(a^\rho)\big|+(1-t^\rho)\big|f^\prime(b^\rho)\big|\big]\,dt \\
				&\hspace{-1cm}\leq \frac{b^\rho - a^\rho}{2}\bigg\{\int_0^{\frac{1}{\sqrt[\rho]{2}}}t^{\rho-1}\left[(1-t^\rho)^\alpha - t^{\rho\alpha}\right]\big[t^\rho\big|f^\prime(a^\rho)\big|+(1-t^\rho)\big|f^\prime(b^\rho)\big|\big]\,dt \\
				&\hspace{2cm}+ \int^1_{\frac{1}{\sqrt[\rho]{2}}}t^{\rho-1}\left[t^{\rho\alpha}-(1-t^\rho)^\alpha\right]\big[t^\rho\big|f^\prime(a^\rho)\big|+(1-t^\rho)\big|f^\prime(b^\rho)\big|\big]\,dt\bigg\} \\
				&\hspace{-1cm} = \int_0^1 g(t)\,dt -2\int_0^{\frac{1}{\sqrt[\rho]{2}}} g(t)\,dt, \quad \mbox{where} \; g(t)= t^{\rho-1}\left[t^{\rho\alpha}-(1-t^\rho)^\alpha\right]\big[t^\rho\big|f^\prime(a^\rho)\big|+(1-t^\rho)\big|f^\prime(b^\rho)\big|\big] \\
				&\hspace{-1cm} = \frac{1}{\rho}\Big[\big|f^\prime(a^\rho)\big|-\big|f^\prime(b^\rho)\big|\Big]\cdot\frac{\alpha}{(\alpha+1)(\alpha+2)} \\
				& \hspace{-.4cm} -2\bigg\{ \big(\big|f^\prime(a^\rho)\big|+\big|f^\prime(b^\rho)\big|\big)\Big[\frac{(\frac{1}{2})^{\alpha+2}}{\alpha+1}+\frac{(\frac{1}{2})^{\alpha+2}}{\alpha+2}+\frac{(\frac{1}{2})^{\alpha+2}}{(\alpha+1)(\alpha+2)}\Big] -\frac{\big|f^\prime(a^\rho)\big|}{(\alpha+1)(\alpha+2)}-\frac{\big|f^\prime(b^\rho)\big|}{\alpha+2}\bigg\}\\
				&\hspace{-1cm} = \frac{b^\rho - a^\rho}{2\rho(\alpha+1)}\Big(1-\frac{1}{2^\alpha}\Big)\big[|f^\prime(a^\rho)|+|f^\prime(b^\rho)|\big].			
\end{align*}
This completes the proof of the theorem. \end{proof}

When $\rho=1$, Theorem~\ref{th4} will reduce to Theorem~3 of \cite{sari1}. If $1$ is in the domain of $f$ and $f$ is differentiable at $1$, then we have the following special case when $\rho \rightarrow 0^+$. 
\[
   \bigg|f(1) - \frac{\alpha\Gamma(\alpha+1)}{2(\ln\frac{b}{a})^\alpha}\Big[H_{a+}^\alpha f(1)+H_{b-}^\alpha f(1)\Big]\bigg| \leq \frac{1}{(\alpha+1)\ln \frac{b}{a}}\Big(1-\frac{1}{2^\alpha}\Big)\big|f^\prime(1)\big|.
\]
where $H_{a+}^\alpha(\cdot)$ and $H_{b-}^\alpha(\cdot)$ are Hadamard fractional integrals defined in Eq.~(\ref{Had1}).
\section{Further inequalities}

In this section, we generalize the results of Jleli et al. \cite{HH14} further. Let $f:[a, b] \to \R$ be a given function, where $0<a<b<\infty$. For the rest of the paper, we define $F(x) := f(x) + f(a+b-x)$. Then it is easy to show that if $f(x)$ is convex on $[a, b]$, $F(x)$ is also convex. The function $F$ has several interesting properties, especially, 
\begin{itemize}
\item $F(x)$ is symmetric to $(a+b)/2$;
\item $F(a) = F(b) = f(a)+f(b)$;
\item $F(\frac{a+b}{2}) = 2f(\frac{a+b}{2})$. 
\end{itemize}
%

\subsection{Hermite-Hadamard type inequalities}
Hermite--Hadamard inequalities can be generalized via Katugampola fractional integrals as follows.
\begin{theorem} 
\label{HH_thm1}
If $f$ is a convex function on $[a, b]$ and $f\in L[a, b]$. Then $F(x)$ is also integrable, and the following inequalities hold     
\begin{equation}
\label{HH_eqn1}
F\Big(\frac{a+b}{2}\Big)\leq \frac{\rho^\alpha \Gamma(\alpha+1)}{2(b^\rho-a^\rho)^\alpha}\left[ \,^\rho I^\alpha_{a+}F(b)+ \,^\rho I^\alpha_{b-}F(a)\right] \leq \frac{F(a)+F(b)}{2}
\end{equation}
with $\alpha > 0$ and $\rho>0$.
\end{theorem}
\begin{proof}
Since $f(x)$ is a convex function on $[a, b]$, we have for $x, y \in [a, b]$ 
\[f\Big(\frac{x+y}{2}\Big)\leq \frac{f(x)+f(y)}{2}.\]
Set $x = ta+(1-t)b$ and $y = (1-t)a+tb$, then 
\begin{equation*}
2f\Big(\frac{a+b}{2}\Big) \leq f(ta+(1-t)b) + f((1-t)a+tb),
\end{equation*}
Using the notation of $F(x)$, we have
\begin{equation}
\label{eqnconvex}
F\Big(\frac{a+b}{2}\Big)\leq F((1-t)a+tb).
\end{equation}
Multiplying both sides of \eqref{eqnconvex} by 
\begin{equation}
\label{factor1}
\frac{((1-t)a+tb)^{\rho-1}}{[b^\rho - ((1-t)a+tb)^\rho]^{1-\alpha}}
\end{equation} 
and integrating the resulting inequality with respect to $t$ over $[0, 1]$, we get 
\begin{eqnarray*}
F\Big(\frac{a+b}{2}\Big) \frac{(b^\rho-a^\rho)^\alpha}{\alpha\rho(b-a)} &\leq& \int_0^1 \frac{((1-t)a+tb)^{\rho-1}}{[b^\rho - ((1-t)a+tb)^\rho]^{1-\alpha}} F((1-t)a+tb) dt\\
&=& \int_a^b \frac{u^{\rho-1}}{(b^\rho - u^\rho)^{1-\alpha}}F(u)\frac{du}{b-a}\\
&=& \frac{\Gamma (\alpha)\rho^{\alpha-1}}{b-a} \,^\rho I^\alpha_{a+}F(b)
\end{eqnarray*}
$i.e.$ 
\begin{equation} 
\label{eqn1}
F\Big(\frac{a+b}{2}\Big) \leq \frac{\Gamma (\alpha+1)\rho^{\alpha}}{(b^\rho-a^\rho)^\alpha} \,^\rho I^\alpha_{a+}F(b).
\end{equation}
Similarly, multiplying both sides of \eqref{eqnconvex} by 
\begin{equation}
\label{factor2}
\frac{((1-t)a+tb)^{\rho-1}}{[ ((1-t)a+tb)^\rho-a^\rho]^{1-\alpha}},
\end{equation} 
and integrating the resulting inequality over $[0, 1]$, we get 
\begin{equation} 
\label{eqn2}
F\Big(\frac{a+b}{2}\Big) \leq \frac{\Gamma (\alpha+1)\rho^{\alpha}}{(b^\rho-a^\rho)^\alpha} \,^\rho I^\alpha_{b-}F(a).
\end{equation}
By adding inequalities \eqref{eqn1} and \eqref{eqn2}, we obtain 
\[F\Big(\frac{a+b}{2}\Big)\leq \frac{\rho^\alpha \Gamma(\alpha+1)}{2(b^\rho-a^\rho)^\alpha}\left[ \,^\rho I^\alpha_{a+}F(b)+ \,^\rho I^\alpha_{b-}F(a)\right].\]
The first inequality of \eqref{HH_eqn1} is proved.

For the second part, since $f$ is a convex function, then for $t \in [0, 1]$, we have 
\begin{equation*}
f(ta+(1-t)b) + f((1-t)a+tb)\leq f(a) + f(b).
\end{equation*}
Using the notation of $F(x)$, we then have
\begin{equation}
\label{eqn3}
F((1-t)a+tb) \leq \frac{F(a)+F(b)}{2}.
\end{equation}
Multiplying both sides of \eqref{eqn3} by factor \eqref{factor1} and integrating the resulting inequality over $[0, 1]$ with respect to $t$, we get 
\begin{equation*}
\frac{\Gamma{(\alpha)} \rho^{\alpha-1}}{b-a}\,^\rho I^\alpha_{a+}F(b) \leq \frac{(b^\rho-a^\rho)^\alpha}{\alpha\rho(b-a)} \frac{F(a)+F(b)}{2} 
\end{equation*}
$i.e.$
\begin{equation}
\label{eqn4}
\frac{\rho^\alpha \Gamma(\alpha+1)}{(b^\rho-a^\rho)^\alpha} \,^\rho I^\alpha_{a+}F(b) \leq \frac{F(a)+F(b)}{2} 
\end{equation}
Similarly, multiplying both sides of inequality \eqref{eqn3} by factor \eqref{factor2} and integrating the resulting inequality over $[0, 1]$, we get 
\begin{equation}
\label{eqn5}
\frac{\rho^\alpha \Gamma(\alpha+1)}{(b^\rho-a^\rho)^\alpha} \,^\rho I^\alpha_{b-}F(a) \leq \frac{F(a)+F(b)}{2} 
\end{equation}
By adding inequality \eqref{eqn4} and \eqref{eqn5}, we obtain 
\[\frac{\rho^\alpha \Gamma(\alpha+1)}{2(b^\rho-a^\rho)^\alpha}\left[ \,^\rho I^\alpha_{a+}F(b)+ \,^\rho I^\alpha_{b-}F(a)\right] \leq \frac{F(a)+F(b)}{2}.\]
The proof is completed.
\end{proof}

\begin{remark}
Theorem \ref{HH_thm1} is a generalization of Hermite-Hadamard inequality. 
\begin{enumerate}
\item Letting $\rho \to 1$ in \eqref{HH_eqn1} and noticing that \begin{eqnarray*}
\lim_{\rho\to 1}\,^\rho I^\alpha_{a+}F(b) &=& \frac{1}{\Gamma(\alpha)} \int_a^b (b-t)^{\alpha-1} F(t) dt = J^\alpha_{a+}f(b) + J^\alpha_{b-}f(a),\\
\lim_{\rho\to 1}\,^\rho I^\alpha_{b-}F(a) &=& \frac{1}{\Gamma(\alpha)} \int_a^b (t-a)^{\alpha-1} F(t) dt = J^\alpha_{b-}f(a) + J^\alpha_{a+}f(b),
\end{eqnarray*}
we immediately get the Riemann-Liouville form of Hermite-Hadamard inequality \eqref{Sari_HH1} in Theorem \ref{sarikaya_thm1}.
\item If $f$ is also symmetric to $\frac{a+b}{2}$, then $F(x) = f(x)+ f(a+b-x) = 2f(x)$, and the inequality \eqref{HH_eqn1} becomes
\begin{equation}
f\Big(\frac{a+b}{2}\Big)\leq \frac{\rho^\alpha \Gamma(\alpha+1)}{2(b^\rho-a^\rho)^\alpha}\left[ \,^\rho I^\alpha_{a+}f(b)+ \,^\rho I^\alpha_{b-}f(a)\right] \leq \frac{f(a)+f(b)}{2}.
\end{equation} 
We can get inequality \eqref{Sari_HH1} directly by letting $\rho \to 1$.
\end{enumerate}
\end{remark}

On the other hand, letting $\rho \to 0^+$ in inequality \eqref{HH_eqn1}, we get the following Hermite-Hadamard inequality for Hadamard fractional integrals.

\begin{corollary} 
If $f$ is a convex function on $[a, b]$ and $f\in L[a, b]$. Then $F(x)$ is also convex and  $F\in L[a, b],$ and the following equalities hold
\begin{equation}
F\Big(\frac{a+b}{2}\Big)\leq \frac{ \Gamma(\alpha+1)}{2(\ln (b/a))^\alpha}\left[ H^\alpha_{a+}F(b)+H^\alpha_{b-}F(a)\right] \leq \frac{F(a)+F(b)}{2}
\end{equation}
with $\alpha > 0$ and $\rho >0$.
\end{corollary}

In order to prove Theorem \ref{thm3.7}, we need the following lemma.
\begin{lemma} 
\label{lemma1}
Let $f: [a, b] \to \R$ be a differentiable mapping on $(a, b)$ with $a<b$. If $f' \in L[a, b]$, then $F$ is also differentiable and $F' \in L[a, b]$, and the following equality holds:
\begin{eqnarray}
\label{eqn6}
\frac{F(a)+F(b)}{2} - \frac{\rho^\alpha \Gamma(\alpha+1)}{2(b^\rho-a^\rho)^\alpha}\left[ \,^\rho I^\alpha_{a+}F(b)+ \,^\rho I^\alpha_{b-}F(a)\right]  = \frac{b-a}{2(b^\rho-a^\rho)^\alpha} \int_0^1 \K(t) F'((1-t)a + bt) dt 
\end{eqnarray}
with $\alpha > 0$ and $\rho >0$, where $\K(t) = [((1-t)a + bt)^\rho - a^\rho]^\alpha - [b^\rho-((1-t)a + bt)^\rho]^\alpha$.
\end{lemma}
\begin{proof}
Note that 
\begin{eqnarray*}
I &=& \int_0^1 \K(t) F'((1-t)a + bt) dt \\
&=& \int_0^1  [((1-t)a + bt)^\rho - a^\rho]^\alpha F'((1-t)a + bt) dt  - \int_0^1  [b^\rho-((1-t)a + bt)^\rho]^\alpha F'((1-t)a + bt) dt \\
&=& I_1+I_2. 
\end{eqnarray*}
Integrating by parts, we get 
\begin{eqnarray}
\label{eqn7}
I_1 &=& \int_0^1  [((1-t)a + bt)^\rho - a^\rho]^\alpha F'((1-t)a + bt) dt = \frac{1}{b-a} \int_a^b [u^\rho - a^\rho]^\alpha d F(u) \\
&=&\left[ \frac{(u^\rho - a^\rho)^\alpha F(u)}{b-a} \right]_a^b - \frac{\alpha \rho}{b-a} \int_a^b \frac{u^{\rho-1}}{(u^\rho - a^\rho)^{1-\alpha}}F(u) du \nonumber\\
&=& \frac{(b^\rho - a^\rho)^\alpha }{b-a}F(b) - \frac{\Gamma(\alpha+1) \rho^\alpha}{b-a} \,^\rho I_{b-}^\alpha F(a). \nonumber 
\end{eqnarray}
Similarly, 
\begin{eqnarray}
\label{eqn8}
I_2 &=& -\int_0^1  [b^\rho-((1-t)a + bt)^\rho ]^\alpha F'((1-t)a + bt) dt \\
&=& \frac{(b^\rho - a^\rho)^\alpha }{b-a}F(a) - \frac{\Gamma(\alpha+1) \rho^\alpha}{b-a} \,^\rho I_{a+}^\alpha F(b). \nonumber 
\end{eqnarray}
By adding \eqref{eqn7} and \eqref{eqn8}, we get 
\begin{eqnarray*}
I = \frac{(b^\rho - a^\rho)^\alpha }{b-a}[F(a)+F(b)] - \frac{\Gamma(\alpha+1) \rho^\alpha}{b-a} [\,^\rho I_{b-}^\alpha F(a) +\,^\rho I_{a+}^\alpha F(b) ]. 
\end{eqnarray*}
Then, multiplying both sides by $\frac{b-a}{2(b^\rho-a^\rho)^\alpha}$ we obtain equality \eqref{eqn6}.
\end{proof}


We are now ready to prove the following Hermite--Hadamard type inequality.
\begin{theorem} 
\label{thm3.7}
Let $f: [a, b] \to \R$ be a differentiable mapping on $(a, b)$ with $a<b$ and $f'\in L[a, b]$. Then $F$ is also differentiable and $F' \in L[a, b]$. If $|f'|$ is convex on $[a, b]$, then the following inequality holds:
\begin{eqnarray}
\label{eqn9}
\left| \frac{F(a)+F(b)}{2} - \frac{\rho^\alpha \Gamma(\alpha+1)}{2(b^\rho-a^\rho)^\alpha}\left[ \,^\rho I^\alpha_{a+}F(b)+ \,^\rho I^\alpha_{b-}F(a)\right] \right| \leq  \frac{b-a}{2(b^\rho-a^\rho)^\alpha} \int_0^1 |\K(t)|dt \, ( |f'(a)| + |f'(b)|) 
\end{eqnarray}
with $\alpha > 0$ and $\rho >0$, where $\K(t) = [((1-t)a + bt)^\rho - a^\rho]^\alpha - [b^\rho-((1-t)a + bt)^\rho]^\alpha$.
\end{theorem}
\begin{proof}
Notice that $F'(x) = f'(x)-f'(a+b-x)$. By the convexity of $|f'|$, we have 
\begin{eqnarray}
\label{eqn10}
|F'((1-t)a + bt)| &=& |f'((1-t)a + bt) - f'(ta +(1-t)b)| \nonumber\\
&\leq& (1-t)|f'(a)| + t|f'(b)| + t|f'(a)| + (1-t)|f'(b)|  \\
&=& |f'(a)| + |f'(b)|.\nonumber
\end{eqnarray}
By inequalities \eqref{eqn6} and \eqref{eqn10}, we get 
\begin{eqnarray*}
& &\left| \frac{F(a)+F(b)}{2} - \frac{\rho^\alpha \Gamma(\alpha+1)}{2(b^\rho-a^\rho)^\alpha}\left[ \,^\rho I^\alpha_{a+}F(b)+ \,^\rho I^\alpha_{b-}F(a)\right] \right| \\
&\leq& \frac{b-a}{2(b^\rho-a^\rho)^\alpha} \int_0^1 |\K(t)| |F'((1-t)a + bt)| dt\\
&\leq& \frac{b-a}{2(b^\rho-a^\rho)^\alpha} \int_0^1 |\K(t)|dt \, ( |f'(a)| + |f'(b)|).
\end{eqnarray*}
\end{proof}

\begin{remark}
In Theorem \ref{thm3.7}, by letting $\rho \to 1$, inequality \eqref{eqn6} becomes inequality \eqref{Sari_HH2} of  Theorem \ref{sarikaya_thm2}. As 
\begin{eqnarray*}
 \lim_{\rho \to 1}\int_0^1 |\K(t)|dt &=& (b-a)^\alpha \int_0^1 \left| t^\alpha -(1-t)^\alpha \right| dt\\
&=&(b-a)^\alpha \left[ \int_0^{\frac{1}{2}} ((1-t)^\alpha - t^\alpha )dt + \int_{\frac{1}{2}}^{1} (t^\alpha -(1-t)^\alpha)dt\right]\\
&=& \frac{2 (b-a)^\alpha}{\alpha+1}\left( 1-\frac{1}{2^\alpha}\right).
\end{eqnarray*}
\end{remark}

On the other hand, by letting $\rho \to 0^+$ in equality \eqref{eqn6}, we get the following result for Hadamard fractional integrals.
\begin{corollary}
Let $f: [a, b] \to \R$ be a differentiable mapping on $(a, b)$ with $a<b$ and $f'\in L[a, b]$. Then $F$ is also differentiable and $F' \in L[a, b]$. If $|f'|$ is convex on $[a, b]$, then the following inequality holds:
\begin{eqnarray}
\left| \frac{F(a)+F(b)}{2} - \frac{\rho^\alpha \Gamma(\alpha+1)}{2(\ln(\frac{b}{a}))^\alpha}\left[  H^\alpha_{a+}F(b)+ H^\alpha_{b-}F(a)\right] \right| \leq  \frac{b-a}{2(\ln(\frac{b}{a}))^\alpha} \int_0^1 |\K(t)|dt \, ( |f'(a)| + |f'(b)|)
\end{eqnarray}
with $\alpha>0$, where $\K(t) = [((1-t)a + bt)^\rho - a^\rho]^\alpha - [b^\rho-((1-t)a + bt)^\rho]^\alpha$.
\end{corollary}

\subsection{Hermite-Hadamard-Fej{\'e}r type inequalities}
The Hermite--Hadamard--Fej{\'e}r inequalities can also be generalized via Katugampola fractional integrals as follows.
\begin{theorem} 
\label{thm3.10}
Let $f:[a, b]\to \R$ be convex function with $a<b$ and $f\in L[a, b]$. Then $F(x)$ is also convex and  $F\in L[a, b].$ If $g:[a, b]\to \R$ is nonnegative and integrable, then the following inequalities hold:
\begin{eqnarray}
\label{eqn11}
F\Big(\frac{a+b}{2}\Big)\left[ \,^\rho I^\alpha_{a+}g(b)+\,^\rho I^\alpha_{b-}g(a)\right] \leq \left[ \,^\rho I^\alpha_{a+}(gF)(b)+\,^\rho I^\alpha_{b-}(gF)(a)\right]  \leq \frac{F(a)+F(b)}{2}\left[ \,^\rho I^\alpha_{a+}g(b)+\,^\rho I^\alpha_{b-}g(a)\right]
\end{eqnarray}
with $\alpha > 0$ and $\rho>0$.
\end{theorem}
\begin{proof}
Since $f$ is convex on $[a, b]$, for all $t \in [0, 1]$, we have 
\begin{equation*}
2f\Big(\frac{a+b}{2}\Big) \leq f(ta+(1-t)b) + f((1-t)a+tb),
\end{equation*}
That is
\begin{equation}
\label{eqnconvex1}
F\Big(\frac{a+b}{2}\Big)\leq F((1-t)a+tb).
\end{equation}
Multiplying both sides of \eqref{eqnconvex1} by 
\begin{equation}
\label{factor11}
\frac{((1-t)a+tb)^{\rho-1}}{[b^\rho - ((1-t)a+tb)^\rho]^{1-\alpha}}g((1-t)a+tb)
\end{equation} 
and integrating the resulting inequality with respect to $t$ over $[0, 1]$, we get 
\begin{eqnarray*}
\frac{\rho^{\alpha-1}\Gamma(\alpha)}{b-a}  \,^\rho I^\alpha_{a+}g(b) F(\frac{a+b}{2}) &\leq& \int_0^1 \frac{((1-t)a+tb)^{\rho-1}}{[b^\rho - ((1-t)a+tb)^\rho]^{1-\alpha}} g((1-t)a+tb) F((1-t)a+tb) dt\\
&=& \int_a^b \frac{u^{\rho-1}}{(b^\rho - u^\rho)^{1-\alpha}}(gF)(u)\frac{du}{b-a}\\
&=& \frac{\Gamma (\alpha)\rho^{\alpha-1}}{b-a} \,^\rho I^\alpha_{a+}(gF)(b),
\end{eqnarray*}
$i.e.$ 
\begin{equation} 
\label{eqn12}
F(\frac{a+b}{2}) \,^\rho I^\alpha_{a+}g(b) \leq  \,^\rho I^\alpha_{a+}(gF)(b).
\end{equation}
Similarly, we have 
\begin{equation} 
\label{eqn13}
F(\frac{a+b}{2})\,^\rho I^\alpha_{b-}g(a) \leq  \,^\rho I^\alpha_{b-}(gF)(a).
\end{equation}
By adding inequalities \eqref{eqn12} and \eqref{eqn13}, we obtain 
\[F(\frac{a+b}{2})\left[ \,^\rho I^\alpha_{a+}g(b)+\,^\rho I^\alpha_{b-}g(a)\right] \leq \left[ \,^\rho I^\alpha_{a+}(gF)(b)+\,^\rho I^\alpha_{b-}(gF)(a)\right]\]
The first inequality of \eqref{eqn11} is proved.

For the second inequality, since $f$ is a convex function, then for all $t \in [0, 1]$, we have 
\begin{equation*}
f(ta+(1-t)b) + f((1-t)a+tb)\leq f(a) + f(b),
\end{equation*}
which can be rewritten as
\begin{equation}
\label{eqn14}
F((1-t)a+tb) \leq \frac{F(a)+F(b)}{2}.
\end{equation}
Multiplying both sides of \eqref{eqn14} by factor \eqref{factor11} and integrating over $[0, 1]$ with respect to $t$, we get 
\begin{equation*}
\frac{\Gamma{(\alpha)} \rho^{\alpha-1}}{b-a}\,^\rho I^\alpha_{a+}(gF)(b) \leq \frac{\Gamma{(\alpha)} \rho^{\alpha-1}}{b-a} \,^\rho I^\alpha_{a+}g(b) \frac{F(a)+F(b)}{2} 
\end{equation*}
$i.e.$
\begin{equation}
\label{eqn15}
 \,^\rho I^\alpha_{a+}(gF)(b) \leq \frac{F(a)+F(b)}{2} \,^\rho I^\alpha_{a+}g(b)
\end{equation}
Similarly, we have 
\begin{equation}
\label{eqn16}
 \,^\rho I^\alpha_{b-}(gF)(a) \leq \frac{F(a)+F(b)}{2} \,^\rho I^\alpha_{b-}g(a)
\end{equation}
Adding inequality \eqref{eqn15} and \eqref{eqn16}, we obtain 
\[ \left[ \,^\rho I^\alpha_{a+}(gF)(b)+\,^\rho I^\alpha_{b-}(gF)(a)\right] \leq \frac{F(a)+F(b)}{2}\left[ \,^\rho I^\alpha_{a+}g(b)+\,^\rho I^\alpha_{b-}g(a)\right].\]
The proof is completed.
\end{proof}

\begin{remark}
Theorem \ref{thm3.10} is a generalization of Hermite--Hadamard--Fej{\'e}r inequalities \cite{Iscan1}. 
\begin{enumerate}
\item If $f$ is symmetric to $\frac{a+b}{2}$, then $F(x) = f(x)+ f(a+b-x) = 2f(x)$, inequality \eqref{eqn11} becomes
\begin{eqnarray}
f(\frac{a+b}{2})\left[ \,^\rho I^\alpha_{a+}g(b)+\,^\rho I^\alpha_{b-}g(a)\right] \leq \left[ \,^\rho I^\alpha_{a+}(gf)(b)+\,^\rho I^\alpha_{b-}(gf)(a)\right]  \leq \frac{f(a)+f(b)}{2}\left[ \,^\rho I^\alpha_{a+}g(b)+\,^\rho I^\alpha_{b-}g(a)\right]
\end{eqnarray}
with $\alpha > 0$ and $\rho>0$.
\item If we take $g(x) =1$ in inequality \eqref{eqn11}, then it becomes inequality \eqref{HH_eqn1} of Theorem \ref{HH_thm1}.
\item If $g(x)$ is symmetric to $(a+b)/2$, then letting $\rho \to 1$, inequality \eqref{eqn11} becomes inequality \eqref{Iscan_HH1} of Theorem \ref{Iscan_thm1}. Since 
\begin{eqnarray*}
\lim_{\rho\to1} \,^\rho I^\alpha_{a+}(gF)(b) &=& \lim_{\rho\to1} \frac{\rho^{1-\alpha}}{\Gamma(\alpha)} \int_a^b \frac{t^{\rho-1}}{(b^\rho-t^\rho)^{1-\alpha}}(gF)(t)dt\\
&=& \frac{1}{\Gamma(\alpha)} \int_a^b (b-t)^{\alpha-1}g(t)f(t)dt+\frac{1}{\Gamma(\alpha)} \int_a^b (b-t)^{\alpha-1}
g(t)f(a+b-t)dt\\
&=&J^\alpha_{a+}(gf)(b) + J^\alpha_{b-}(gf)(a)
\end{eqnarray*}
and similarly$$\lim_{\rho\to1} \,^\rho I^\alpha_{b-}(gF)(a)=J^\alpha_{b-}(gf)(a) + J^\alpha_{a+}(gf)(b).$$
\end{enumerate}
\end{remark}

To prove the inequality in Theorem \ref{thm3.14}, we need the following lemma.
\begin{lemma} 
\label{lemma2}
Let $f: [a, b]\to \R$ be a differentiable mapping on $(a, b)$ with $0\leq a<b$ and $f'\in L[a, b]$. Then $F(x)$ is also differentiable and  $F' \in L[a, b].$ If $g:[a, b]\to \R$ is integrable, then the following equality holds:
\begin{eqnarray}
\label{eqn17}
 & & \frac{F(a)+F(b)}{2}\left[ \,^\rho I^\alpha_{a+}g(b)+\,^\rho I^\alpha_{b-}g(a)\right] -\left[ \,^\rho I^\alpha_{a+}(gF)(b)+\,^\rho I^\alpha_{b-}(gF)(a)\right] \\
 & &\hspace{4cm}=  \frac{\rho^{1-\alpha}}{2\Gamma(\alpha)} \int_a^b \left[ \int_a^t \G(s) g(s)ds - \int_t^b \G(s) g(s)ds \right]F'(t) dt \nonumber 
\end{eqnarray}
with $\alpha>0$ and $\rho>0$, where \[\G(s) = \frac{s^{\rho-1}}{(b^\rho-s^\rho)^{1-\alpha}} + \frac{s^{\rho-1}}{(s^\rho-a^\rho)^{1-\alpha}}. \]
\end{lemma}
\begin{proof}
Note that
\begin{eqnarray*}
I &=& \int_a^b \left[ \int_a^t \G(s) g(s)ds - \int_t^b \G(s) g(s)ds \right]F'(t) dt\\
&=& \int_a^b \int_a^t \G(s) g(s)dsF'(t) dt - \int_a^b\int_t^b \G(s) g(s)ds F'(t) dt\\
&=& I_1+I_2.
\end{eqnarray*}
Integrating by parts, we get 
\begin{eqnarray}
\label{eqn18}
I_1 &=& \int_a^b \int_a^t \G(s) g(s)dsF'(t) dt \\
&=& \left[ \int_a^t \G(s)g(s)ds F(t)\right]_a^b -\int_a^b\G(t)g(t)F(t)dt \nonumber \\
&=& \int_a^b \G(s)g(s)ds F(b)-\int_a^b\G(t)g(t)F(t)dt \nonumber\\
&=& \Gamma(\alpha)\rho^{\alpha-1}  [\,^\rho I_{a+}^\alpha g(b) + \,^\rho I_{b-}^\alpha g(a)]F(b)  -\Gamma(\alpha)\rho^{\alpha-1} [\,^\rho I_{a+}^\alpha (gF)(b) + \,^\rho I_{b-}^\alpha (gF)(a)], \nonumber
\end{eqnarray}
and similarly
\begin{eqnarray}
\label{eqn19}
I_2 &=& -\int_a^b \int_t^b \G(s) g(s)dsF'(t) dt \\
&=& \Gamma(\alpha)\rho^{\alpha-1}  [\,^\rho I_{a+}^\alpha g(b) + \,^\rho I_{b-}^\alpha g(a)]F(a)  -\Gamma(\alpha)\rho^{\alpha-1} [\,^\rho I_{a+}^\alpha (gF)(b) + \,^\rho I_{b-}^\alpha (gF)(a)]. \nonumber
\end{eqnarray}
From \eqref{eqn18} and \eqref{eqn19}, we get
\begin{eqnarray*}
I = \Gamma(\alpha)\rho^{\alpha-1}  [\,^\rho I_{a+}^\alpha g(b) + \,^\rho I_{b-}^\alpha g(a)](F(a)+F(b))  - 2\Gamma(\alpha)\rho^{\alpha-1} [\,^\rho I_{a+}^\alpha (gF)(b) + \,^\rho I_{b-}^\alpha (gF)(a)].
\end{eqnarray*}
Then, multiplying both sides by $ \frac{\rho^{1-\alpha}}{2\Gamma(\alpha)}$, we get the conclusion.
\end{proof}

With Lemma \ref{lemma2}, we have the following Hermite--Hadamard--Fej{\'e}r type inequality.
\begin{theorem} 
\label{thm3.14}
Let $f: [a, b] \to \R$ be a differentiable mapping on $(a, b)$ and $f'\in L[a, b]$ with $0\leq a<b$. Then $F(x)$ is also differentiable and $F'\in L[a, b]$. If $|f'|$ is convex on $[a, b]$ and $g:[a, b]\to \R$ is continuous, then the following inequality holds:
\begin{eqnarray}
\label{eqn20}
 & & \left| \frac{F(a)+F(b)}{2}\left[ \,^\rho I^\alpha_{a+}g(b)+\,^\rho I^\alpha_{b-}g(a)\right] -\left[ \,^\rho I^\alpha_{a+}(gF)(b)+\,^\rho I^\alpha_{b-}(gF)(a)\right] \right| \\
 & & \hspace{5cm}\leq \frac{(b-a) ||g||_\infty}{\rho^\alpha \Gamma(\alpha+1)}\big( |f'(a)| + |f'(b)|\big)\int_0^1 |\K(t)|dt \,  \nonumber
\end{eqnarray}
with $\alpha>0$ and $\rho>0$, where $||g||_\infty = \sup_{t\in [a, b]} |g(x)|$, and \[ \K(t) = [((1-t)a + bt)^\rho - a^\rho]^\alpha - [b^\rho-((1-t)a + bt)^\rho]^\alpha\] as defined in Lemma \ref{lemma1}.
\end{theorem}
\begin{proof}
Notice that $F'(t) = f'(t)- f'(a+b-t)$, and by the convexity of $|f'(t)|$, we have 
\begin{eqnarray*}
|F'(t)| &=& |f'(t)- f'(a+b-t)| \leq |f'(t)|- |f'(a+b-t)|  \\
&=& |f'(\frac{b-t}{b-a}a+ \frac{t-a}{b-a}b)|+ |f'(\frac{t-a}{b-a}a+ \frac{b-t}{b-a}b)| \\
&\leq& \frac{b-t}{b-a} |f'(a)|+\frac{t-a}{b-a}|f'(b)| + \frac{t-a}{b-a}|f'(a)| + \frac{b-t}{b-a}|f'(b)| \\
&=& |f'(a)|+|f'(b)|. 
\end{eqnarray*}
Also 
\begin{eqnarray*}
& & \int_a^t \G(s)ds - \int_t^b \G(s)ds \\
&=& \int_a^t \left( \frac{s^{\rho-1}}{(b^\rho-s^\rho)^{1-\alpha}} + \frac{s^{\rho-1}}{(s^\rho-a^\rho)^{1-\alpha}} \right)ds - \int_t^b \left( \frac{s^{\rho-1}}{(b^\rho-s^\rho)^{1-\alpha}} + \frac{s^{\rho-1}}{(s^\rho-a^\rho)^{1-\alpha}} \right)ds \\
&=& \left[  -\frac{(b^\rho -s^\rho)^\alpha}{\alpha \rho} + \frac{(s^\rho-a^\rho)^\alpha}{\alpha \rho}\right]^t_a - \left[  -\frac{(b^\rho -s^\rho)^\alpha}{\alpha \rho} + \frac{(s^\rho-a^\rho)^\alpha}{\alpha \rho}\right]^b_t \\
&=& \frac{2}{\alpha \rho} [(t^\rho-a^\rho)^\alpha -(b^\rho - t^\rho)^\alpha].
\end{eqnarray*}
Hence by Lemma \ref{lemma2}, 
\begin{eqnarray*}
& & \left| \frac{F(a)+F(b)}{2}\left[ \,^\rho I^\alpha_{a+}g(b)+\,^\rho I^\alpha_{b-}g(a)\right] -\left[ \,^\rho I^\alpha_{a+}(gF)(b)+\,^\rho I^\alpha_{b-}(gF)(a)\right] \right| \\
& \leq&  \frac{\rho^{1-\alpha}}{2\Gamma(\alpha)} \int_a^b \left| \int_a^t \G(s) g(s)ds - \int_t^b \G(s) g(s)ds \right| |F'(t)| dt \\
&\leq& \frac{\rho^{1-\alpha} ||g||_\infty}{2\Gamma(\alpha)} \int_a^b \left| \int_a^t \G(s)ds - \int_t^b \G(s)ds \right| dt (|f'(a)|+|f'(b)|) \\
&=& \frac{ ||g||_\infty}{\rho^\alpha \Gamma(\alpha+1)} \int_a^b \left| (t^\rho-a^\rho)^\alpha -(b^\rho - t^\rho)^\alpha \right| dt (|f'(a)|+|f'(b)|) \\
&=& \frac{ (b-a)||g||_\infty}{\rho^\alpha \Gamma(\alpha+1)} \int_0^1 \left| \K(t) \right| dt (|f'(a)|+|f'(b)|)
\end{eqnarray*}
where $\K(t) = [((1-t)a + bt)^\rho - a^\rho]^\alpha - [b^\rho-((1-t)a + bt)^\rho]^\alpha$.
\end{proof}

\begin{remark}
In Theorem \ref{thm3.14}, 
\begin{enumerate}
\item if we take $g(x) =1$ in inequality \eqref{eqn20}, then it becomes inequality \eqref{eqn9} in Theorem \ref{thm3.7}.
\item if, in addition, $g(x)$ is symmetric to $(a+b)/2$, letting $\rho \to 1$, inequality \eqref{eqn20} becomes inequality \eqref{Iscan_HH2} of Theorem \ref{Iscan_thm2}.
\end{enumerate}
\end{remark}



%

%
%
%
%
%
\section*{Acknowledgement}
The research was partially supported by the U.S. Army Research Office grant W911NF-15-1-0537.
%
%
%
%
%

%
%
%
\bibliographystyle{elsarticle-num}
%
%
%

%
\end{document}